\documentclass[12pt]{amsart}

\usepackage{amsthm,amsfonts,amsmath,amssymb,latexsym,epsfig,mathrsfs,yfonts,marvosym,amscd}
\usepackage[usenames]{color}
\usepackage{graphicx}
\usepackage[all]{xy}
\usepackage{epsfig}
\usepackage{epic}
\usepackage{eepic}
\usepackage{hyperref}
\usepackage{dsfont}\let\mathbb\mathds
\usepackage[english]{babel}

\DeclareMathAlphabet\oldmathcal{OMS}        {cmsy}{b}{n}
\SetMathAlphabet    \oldmathcal{normal}{OMS}{cmsy}{m}{n}
\DeclareMathAlphabet\oldmathbcal{OMS}       {cmsy}{b}{n}
 
\usepackage{eucal}

\usepackage[active]{srcltx}

\newtheorem{theorem}{Theorem}[section]
\newtheorem{lemma}[theorem]{Lemma}
\newtheorem{proposition}[theorem]{Proposition}

\newenvironment{example}{\medskip \refstepcounter{theorem}
\noindent  {\bf Example \thetheorem}.\rm}{\,}
\newenvironment{remark}{\medskip \refstepcounter{theorem}
\newcommand     {\comment}[1]   {}
\newcommand{\mute}[2] {}
\newcommand     {\printname}[1] {}

\noindent  {\bf Remark \thetheorem}.\rm}{\,}
\def\<{\langle}

\def\>{\rangle}

\def\BOne{{\mathchoice {\rm 1\mskip-4mu l} {\rm 1\mskip-4mu l}
                          {\rm 1\mskip-4.5mu l} {\rm 1\mskip-5mu l}}}

\def\fract#1#2{\raise4pt\hbox{$ #1 \atop #2 $}}
\def\decdnar#1{\phantom{\hbox{$\scriptstyle{#1}$}}
\left\downarrow\vbox{\vskip15pt\hbox{$\scriptstyle{#1}$}}\right.}

\def\bbb{{\mathbb B}}
\def\bbc{{\mathbb C}}

\def\bbi{{\mathbb I}}

\def\bbp{{\mathbb P}}

\def\bbr{{\mathbb R}}

\def\bbz{{\mathbb Z}}

\def\gra{\alpha}
\def\grb{\beta}

\def\grg{\gamma}
\def\gri{\iota}

\def\grl{\lambda}

\def\gro{\omega}

\def\grz{\zeta}
\def\grD{\Delta}

\def\grG{\Gamma}

\def\grL{\Lambda}

\def\bfa{{\bf a}}
\def\bfb{{\bf b}}

\def\bfd{{\bf d}}

\def\bfw{{\bf w}}

\def\cald{{\mathcal D}}

\def\calf{{\mathcal F}}
\def\calg{{\mathcal G}}

\def\cali{{\mathcal I}}

\def\cals{{\oldmathcal S}}

\def\la#1{\hbox to #1pc{\leftarrowfill}}
\def\ra#1{\hbox to #1pc{\rightarrowfill}}
\def\calz{{\oldmathcal Z}}

\def\gn{{\mathfrak n}}
\def\go{{\mathfrak o}}

\def\gt{{\mathfrak t}}
\def\gu{{\mathfrak u}}

\def\gA{{\mathfrak A}}

\def\gC{{\mathfrak C}}

\def\gT{{\mathfrak T}}

\def\hook{\mathbin{\hbox to 6pt{%
                 \vrule height0.4pt width5pt depth0pt
                 \kern-.4pt
                 \vrule height6pt width0.4pt depth0pt\hss}}}

\def\cJ{\hat{J}}

\def\lcm{{\rm lcm}}

\def\12{\xi_{k_1,k_2}}
\def\m5{M^5_{k_1,k_2}}

\begin{document}

\title{Sasakian Manifolds with Perfect Fundamental Groups}

\author{Charles P. Boyer and Christina W. T{\o}nnesen-Friedman}
\address{Charles P. Boyer, Department of Mathematics and Statistics,
University of New Mexico, Albuquerque, NM 87131.}
\email{cboyer@math.unm.edu} 
\address{Christina W. T{\o}nnesen-Friedman, Department of Mathematics, Union
College, Schenectady, NY 12308, USA } \email{tonnesec@union.edu}

\thanks{Charles Boyer was partially supported by a grant \#245002 from the Simons Foundation, and Christina T{\o}nnesen-Friedman was partially supported by a grant \#208799 from the Simons Foundation}

\begin{abstract}
Using the Sasakian join construction with homology 3-spheres, we give a countably infinite number of examples of Sasakian manifolds with perfect fundamental group in all odd dimensions $\geq 3$. These have extremal Sasaki metrics with constant scalar curvature. Moreover, we present further examples of both Sasaki-Einstein and Sasaki-$\eta$-Einstein metrics.
\end{abstract}

\keywords{Extremal Sasaki metrics, $\eta$-Einstein metrics, orbifolds, complete intersections, homology spheres}

\subjclass[2000]{Primary: 53D42; Secondary:  53C25}

\maketitle

\markboth{Sasakian Manifolds with Perfect Fundamental Groups}{Charles P. Boyer and Christina W. T{\o}nnesen-Friedman}

\bigskip
\centerline{Dedicated to Professor Augustin Banyaga on the occasion of his 65th birthday}
\bigskip

\section{Introduction}
Until recently, except for the rather obvious examples, little seemed to be known about Sasakian manifolds with non-trivial fundamental group; see, however, \cite{Che11}. In this note we construct many examples of Sasakian manifolds with a perfect fundamental group. The examples we present all have extremal Sasaki metrics (in the sense of Calabi), and have constant scalar curvature. When the Sasaki cone has dimension greater than one, the Openness Theorem of \cite{BGS06} implies the existence of other extremal Sasaki metrics which generally do not have constant scalar curvature. We leave for future work the explicit construction of extremal non constant scalar curvature Sasaki metrics. Our main purpose here is to prove the following theorems:

\begin{theorem}\label{thmA}
For each odd dimension $\geq 3$ there exists a countable infinity of Sasakian manifolds with a perfect fundamentfinial group which admit Sasaki metrics with constant scalar curvature. Furthermore, there is an infinite number of such Sasakian manifolds that have the integral cohomology ring of $S^2\times S^{2r+1}$.
\end{theorem}

\begin{theorem}\label{MaMbthm}
There exist a countably infinite number of contact 5-manifolds with perfect fundamental group and the integral cohomology ring of $S^2\times S^3$ that admit Sasaki metrics with constant scalar curvature. Moreover, there are such contact 5-manifolds that admit a ray of Sasaki-$\eta$-Einstein metrics.
\end{theorem}

\begin{theorem}\label{sasetaein}
There exist negative Sasaki-$\eta$-Einstein (hence, Lorentzian Sasaki-Einstein) 7-manifolds with perfect fundamental group and arbitrary second Betti number $\geq 1$.
\end{theorem}

\begin{theorem}\label{sethm}
Let $N$ be a simply connected quasi-regular Sasaki-Einstein $(2r+1)$-dimensional manifold with Fano index $I_F$ and order $\upsilon$. Then if $\gcd(30I_F,\upsilon)=1$, the join $S^3/\bbi^*\star_{1,I_F}N$ is a smooth $(2r+3)$-manifold with perfect fundamental group and admits a Sasaki-Einstein metric. Such examples exist with rational cohomology ring of $S^2\times S^{2r+1}$ for all $r\geq 1$. Moreover, $S^3/\bbi^*\star_{1,2}S^3$ is Sasaki-Einstein and has integral cohomology ring of $S^2\times S^3$.
\end{theorem}

Many examples of Sasaki-Einstein manifolds with perfect fundamental group can be worked out by choosing $N$ appropriately, as for example in Chapters 11 and 13 of \cite{BG05}, and in particular the 5-dimensional Sasaki-Einstein manifolds in \cite{BGK05,Kol05b,Kol04,BN09}.

\section{Preliminaries on Sasakian Geometry}
Here we give a brief review of Sasakian geometry referring to \cite{BG05} for details and further development. Sasakian  geometry can be thought of roughly as the odd dimensional version of K\"ahlerian geometry. Its relation to contact geometry mimics the relation of K\"ahler geometry to symplectic geometry. 

\subsection{Sasakian Structures}
A Sasakian structure on a smooth manifold $M$ of dimension $2n+1$  consists of a contact 1-form $\eta$ together with its Reeb vector field $\xi$ which satisfies $\eta(\xi)=1$ and $\xi\hook d\eta=0$, an endomorphism field $\Phi$ which annihilates $\xi$ and defines a strictly pseudoconvex CR structure $(\cald,J)$ satisfying $\cald=\ker\eta$ and $J=\Phi |_\cald$, and finally a compatible Riemannian metric $g$ defined by the equation 
\begin{equation}\label{sasmetric}
g=d\eta\circ(\Phi\otimes \BOne)+\eta\otimes\eta,
\end{equation}
such that $\xi$ is a Killing vector field of $g$. We denote such a Sasakian structure by the quadruple $\cals=(\xi,\eta,\Phi,g)$.  Note that the Reeb vector field $\xi$ generates a one dimensional foliation $\calf_\xi$ of $M$ whose transverse structure is K\"ahler with transverse K\"ahler form $d\eta$. The transverse K\"ahler metric on $\cald$ is $g_T=d\eta\circ (\Phi\otimes\BOne)$. There is a freedom of scaling, namely, given a Sasakian structure $\cals=(\xi,\eta,\Phi,g)$ consider the {\it transverse homothety} defined by sending the Sasakian structure  to 
\begin{eqnarray}\label{transhomot}
\cals &=(\xi,\eta,\Phi,g)\mapsto \cals_a=(a^{-1}\xi,a\eta,\Phi,g_a)~\text{where $a\in\bbr^+$ and} \\ 
&g_a =ag+(a^2-a)\eta\otimes\eta. \notag
\end{eqnarray}
$\cals_a$ is another Sasakian structure which generally is inequivalent to $\cals$; hence, Sasakian structures come in rays. 

When $M$ is compact the closure of any leaf of $\calf_\xi$ is a torus $\gT$ of dimension at least one, and the flow is conjugate to a linear flow on the torus (cf. Theorem 2.6.4 of \cite{BG05}). This implies that for a dense subset of Sasakian structures $\cals$ on a compact manifold the leaves are all compact 1-dimensional manifolds, i.e circles. Such $\cals$ are known as {\it quasiregular} in which case the foliation $\calf_\xi$ comes from a locally free circle action. Assuming this circle action is effective the isotropy subgroups are all finite cyclic groups, and the least common multiple $\upsilon=\upsilon(\cals)$ of theirs orders is an invariant of the Sasakian structure $\cals$, called the {\it order} of $\cals$. Then the quotient space $\calz$ has the structure of a projective algebraic orbifold with an induced K\"ahler form $\gro$ such that $\pi^*\gro=d\eta$ where $\pi$ is the quotient projection. The isotropy subgroups of the local circle action on $M$ give rise to the local uniformizing groups of the orbifold. If the circles comprising the leaves of $\calf_\xi$ all have the same period, $\cals$ is said to be {\it regular}, and the quotient space $\calz$ is a smooth projective algebraic variety with a trivial orbifold structure. A leaf of $\calf_\xi$ that is not quasi-regular is a copy of $\bbr$ in which case $\calf_\xi$ is said to be {\it irregular}. A theorem of Rukimbira \cite{Ruk95a} says that an irregular Sasakian structure can be approximated by quasi-regular ones. A somewhat more general case occurs if we drop the condition that $(\cald,J)$ be a CR structure, but only consider a strictly pseudoconvex {\it almost} CR structure, that is, the almost complex structure $J$ is not necessarily integrable. Then $\cals=(\xi,\eta,\Phi,g)$ is called a {\it K-contact} structure. A quasi-regular contact structure is equivalent to having a compatible K-contact structure \cite{Ruk95a}. All K-contact structures considered in this paper are Sasakian. The (almost) complex structure $\cJ$ on $\calz$ is also related to the (almost) CR structure $J$ on $M$. For any foliate vector field $X$ on $M$ we have $\pi_*\Phi X=\cJ\pi_*X$. We say that $J=\Phi |_\cald$ is the {\it horizontal lift} of $\cJ$.

The flow of the Reeb vector field $\xi$ lies in the center of the automorphism group $\gA\gu\gt(\cals)$ of a Sasakian structure $\cals=(\xi,\eta,\Phi,g)$; hence, any Sasakian structure on a compact $(2n+1)$-dimensional manifold has a $k$-dimensional torus $\gT_k$ in its automorphism group, where $1\leq k\leq n+1$. The subset $\gt_k^+$ of the Lie algebra $\gt_k$ of this torus consisting of vector fields $\xi'$ that satisfy the positivity condition $\eta(\xi')>0$ everywhere on $M$ forms a cone called the {\it Sasaki cone} \cite{BGS06}. It provides a $k$-dimensional family of Sasakian structures associated with $\cals$ all having the same underlying CR structure $(\cald,J)$. There can be many Sasaki cones associated with the same contact structure $\cald$ as seen, for example, in  \cite{Boy10a,Boy10b,BoTo11,BoTo12}. These give rise to bouquets of Sasaki cones which correspond to distinct conjugacy classes of tori in the contactomorphism group $\gC\go\gn(\cald)$. For more on the important infinite dimensional Fr\'echet Lie group $\gC\go\gn(\cald)$ we refer to Banyaga's seminal book \cite{Ban97}.

The conditions on the Riemannian curvature for Sasaki metrics have been very well studied, and we refer to \cite{Bla10,BG05} and references therein for details. It suffices here to mention only some basic facts about the Ricci curvature and $\Phi$ sectional curvature of a Sasaki metric. Let $\cals=(\xi,\eta,\Phi,g)$ be a Sasakian structure on a $(2n+1)$-dimensional manifold, then the Ricci curvature ${\rm Ric}_g$ of $g$ satisfies the identities
\begin{equation}\label{Ricsas}
\xi\hook {\rm Ric}_g=2n\eta, \qquad {\rm Ric}_g|_{\cald\times \cald}={\rm Ric}_T-2g|_{\cald\times \cald}
\end{equation}
where ${\rm Ric}_T$ denotes the Ricci curvature of the transverse K\"ahler metric. One easily sees from these equations that the corresponding scalar curvatures are related by 
\begin{equation}\label{cscsas}
s_g=s_T-2n. 
\end{equation}
If $K$ is the ordinary Riemannian sectional curvature of $g$, then the $\Phi$ sectional curvature of $g$ is defined by $K(X,\Phi X)$ where $X$ is any vector field. Sasaki metrics of constant $\Phi$ sectional curvature are known as {\it Sasakian space forms}. There are three types of Sasakian space forms, those with $K(X,\Phi X)=c>-3,<-3$ and $=-3$. Within each type different constants are related by a transverse homothety. If we have a Sasakian space form of constant $\Phi$ sectional curvature $c$, then after a transverse homothety Equation (\ref{transhomot}), we obtain a  Sasakian space form of constant $\Phi$ sectional curvature $c'=\frac{c+3}{a}-3$. They are the analogs of constant holomorphic sectional curvature in complex geometry. Indeed, under the Boothby-Wang correspondence constant $\Phi$ sectional curvature $c$ corresponds precisely to constant holomorphic transverse sectional curvature $k=c+3$. Assuming the Sasakian manifold is simply connected, in the spherical or positive case ($c>-3$) the transverse K\"ahler structure is that of complex projective space $\bbc\bbp^n$; whereas, in the hyperbolic or negative case ($c<-3$) the transverse K\"ahler structure is that of the complex hyperbolic ball $\bbb^n$. Finally we make note of the Ricci tensor for the transverse K\"ahler structures of constant holomorphic sectional curvature $k$, viz.
\begin{equation}\label{transricciholo}
{\rm Ric}_T=\frac{n+1}{2}kg_T.
\end{equation}

\subsection{The Join Construction}

We shall make use of the {\it join construction} first introduced in \cite{BG00a} in the context of Sasaki-Einstein manifolds, and developed further for general Sasakian structures in \cite{BGO06}, see also Section 7.6.2 of \cite{BG05}. Products of K\"ahlerian manifolds are K\"ahler, but products of Sasakian manifolds do not even have the correct dimension. Nevertheless, one can easily construct new quasi-regular Sasakian manifolds from old quasi-regular ones by constructing circle orbibundles over the product of K\"ahler orbifolds. Let $M_i$ for $i=1,2$ be compact quasi-regular contact manifolds with Reeb vector fields $\xi_i$, respectively. These vector fields generate locally free circle actions on $M_i$ and their quotients are symplectic orbifolds $\calz_i$.  Then the quotient of the product $T^2=S^1\times S^1$ action on $M_1\times M_2$ is $\calz_1\times \calz_2$. Taking primitive symplectic forms $\gro_i$ on $\calz_i$ we consider the symplectic form $\gro_{k_1,k_2}=k_1\gro_1+k_2\gro_2$ on $\calz_1\times \calz_2$ where $k_1,k_2$ are relatively prime positive integers. Then by the orbifold Boothby-Wang construction \cite{BG00a} the total space of the principal circle orbibundle over $\calz_1\times \calz_2$ corresponding to the cohomology class $[\gro_{k_1,k_2}]\in H^2(\calz_1\times \calz_2,\bbz)$ has a natural quasi-regular contact structure whose contact form $\eta_{k_1,k_2}$ satisfies $d\eta_{k_1,k_2}=\pi^*\gro_{k_1,k_2}$ where $\pi$ is the natural orbibundle projection. Moreover, if the base spaces $\calz_i$ are complex orbifolds and the $\gro_i$ K\"ahler forms,  the total space of this orbibundle, denoted by $M_1\star_{k_1,k_2}M_2$, has a natural Sasakian structure. It is called {\it the join} of $M_1$ and $M_2$. Generally, $M_1\star_{k_1,k_2}M_2$ is only an orbifold; however, if $\gcd(\upsilon_1k_2,\upsilon_2k_1)=1$ it will be a smooth manifold, where $\upsilon_i$ is the order the Sasakian structures on $M_i$. So if $M_i$ has Sasakian structures $\cals_i=(\xi_i,\eta_i,\Phi_i,g_i)$ we obtain a new Sasakian structure $\cals_{k_1,k_2}$ on $M_1\star_{k_1,k_2}M_2$ such that the following diagram
\begin{equation}\label{comdia1}
\begin{matrix}  M_1 \times M_2 &&& \\
                          &\searrow && \\
                          \decdnar{\pi_B} && M_1\star_{k_1,k_2}M_2 &\\
                          &\swarrow && \\
                          \calz_1\times \calz_2&&& 
\end{matrix}
\end{equation}
commutes. Here $\pi_B$ is the quotient projection of the $T^2$ torus action, the southeast arrow is the quotient projection by the circle action generated by the vector field 
$\frac{1}{2k_1}\xi_1-\frac{1}{2k_2}\xi_2$ and the southwest arrow is the quotient projection of the Reeb vector field  $\frac{1}{2k_1}\xi_1+\frac{1}{2k_2}\xi_2$ of $M_1\star_{k_1,k_2}M_2$.

Although we are not able to distinguish diffeomorphism types of $M_1\star_{k_1,k_2}M_2$ generally, we can distinguish contact structures by the first Chern class $c_1(\cald)$ of the contact bundle $\cald$. 

\subsection{Extremal Sasakian Structures}
Given a Sasakian structure $\cals=(\xi,\eta,\Phi,g)$ on a compact manifold $M^{2n+1}$ we deform the contact 1-form by $\eta\mapsto \eta(t)=\eta+t\grz$ where $\grz$ is a basic 1-form with respect to the characteristic foliation $\calf_\xi$ defined by the Reeb vector field $\xi.$ Here $t$ lies in a suitable interval containing $0$ and such that $\eta(t)\wedge d\eta(t)\neq 0$. This gives rise to a family of Sasakian structures $\cals(t)=(\xi,\eta(t),\Phi(t),g(t))$ that we denote by ${\mathcal S}(\xi, \bar{J})$ where $\bar{J}$ is the induced complex structure on the normal bundle $\nu(\calf_\xi)=TM/L_\xi$ to the Reeb foliation $\calf_\xi$ which satisfy the initial condition $\cals(0)=\cals$. On the space ${\mathcal S}(\xi, \bar{J})$  we consider the ``energy functional'' $E:{\mathcal S}(\xi,\bar{J})\ra{1.4} \bbr$ defined by
\begin{equation}\label{var}
E(g) ={\displaystyle \int _M s_g ^2 d{\mu}_g ,}\, 
\end{equation}
i.e. the $L^2$-norm of the scalar curvature $s_g$ of the Sasaki metric $g$. Critical points $g$ of this functional are called {\it extremal Sasaki metrics}.  Similar to the K\"ahlerian case, the Euler-Lagrange equations for this functional says \cite{BGS06} that $g$ is critical if and only if the gradient vector field $J{\rm grad}_gs_g$ is transversely holomorphic, so, in particular, Sasakian metrics with constant scalar curvature are extremal. Since the scalar curvature $s_g$ is related to the transverse scalar curvature $s^T_g$ of the transverse K\"ahler metric by $s_g=s_g^T-2n$, a Sasaki metric is extremal if and only if its transverse K\"ahler metric is extremal. Hence, in the quasi-regular case, an extremal K\"ahler orbifold metric lifts to an extremal Sasaki metric, and conversely an extremal Sasaki metric projects to an extremal K\"ahler orbifold metric. Note that the deformation $\eta\mapsto \eta(t)=\eta+t\grz$ not only deforms the contact form, but also deforms the contact structure $\cald$ to an equivalent, isotopic, contact structure. So when we say that a contact structure $\cald$ has an extremal representative, we mean so up to isotopy. Deforming the K\"ahler form within its K\"ahler class corresponds to deforming the contact structure within its isotopy class. 

As mentioned above Sasaki metrics of constant scalar curvature are a special case of extremal Sasaki metrics. We shall abbreviate constant scalar curvature by CSC. A further special case of interest are the so-called Sasaki-$\eta$-Einstein metrics, or simply $\eta$-Einstein (see for example \cite{BGM06,BG05} and references therein). Recall that a Sasakian (or K-contact) structure $\cals=(\xi,\eta,\Phi,g)$ is called {\it $\eta$-Einstein} if there are constants $a,b$ such that 
\begin{equation}\label{etaE}
{\rm Ric}_g=ag +b\eta\otimes \eta
\end{equation}
 where ${\rm Ric}_g$ is the Ricci curvature of $g$. The constants $a,b$ satisfy $a+b=2n$. The scalar curvature $s_g$ of an $\eta$-Einstein metric is constant. Indeed, if the manifold has dimension $2n+1$, then the scalar curvature satisfies $s_g=2n(a+1)$. However, as we can easily see, not every constant scalar curvature Sasaki metric is $\eta$-Einstein. In fact, many of the CSC Sasaki metrics that we construct have a diagonal Ricci tensor consisting of constant diagonal blocks. Notice that if $b=0$ we obtain the more familiar Einstein metric, so $\eta$-Einstein is a generalization of Einstein. In this case the scalar curvature $s_g=2n(2n+1)$ and the transverse scalar curvature is $s_T=4n(n+1)$. Moreover, it follows from Equation (\ref{Ricsas}) that a Sasaki metric $g$ is $\eta$-Einstein if and only if the transverse K\"ahler metric $h$ is Einstein. So one easily sees that a transverse K\"ahler-Einstein metric is negative if and only if $a<-2$, and positive if and only if $a>-2$. We refer to these as negative (positive) $\eta$-Einstein metrics, respectively. Furthermore, given a positive $\eta$-Einstein metric there is a transverse homothety whose resulting metric is Sasaki-Einstein.  Even more is true in dimension $3$: a 3-dimensional Sasaki metric is $\eta$-Einstein if and only if it has constant $\Phi$ sectional curvature \cite{BGM06}. 
 
 It is easy to see that any Sasakian structure gives rise naturally to a Sasaki metric with a Lorentzian signature, see Section 11.8.1 in \cite{BG05}. Moreover, in the Lorentzian signature, one can apply a transverse homothety to a   negative Sasaki-$\eta$-Einstein metric to obtain a Lorentzian Sasaki-Einstein metric. Thus, we can obtain many examples of Lorentzian Sasaki-Einstein metrics \cite{Gom11}. 

\section{Seifert Fibered Homology $3$-Spheres}

Homology spheres are by definition manifolds whose integral homology coincides with that of a sphere. In dimension 3 this is equivalent to the fundamental group being {\it perfect}, that is, it coincides with its commutator subgroup. 
We want the homology 3-spheres that we consider to admit a Sasakian structure; hence, they must be Seifert fibered homology spheres with an effective fixed point free circle action. Here we give a brief review of such homology 3-spheres following \cite{Sav02,LeRa10} and the translation of \cite{Sei33} in \cite{SeTh80}. It is known that the binary icosahedral group $\bbi^*$ is the only non-trivial finite perfect subgroup of $SU(2)$ (see page 181 in \cite{Wol67}). It is a double cover of the simple group $\bbi$ of order $60$, the icosahedral group. Moreover, it follows from Perelman's proof of the Poincar\'e conjecture that up to diffeomorphism the only compact 3-manifold with a non-trivial finite perfect fundamental group is the celebrated Poincar\'e sphere $S^3/\bbi^*$. The remainder of the Seifert fibered homology 3-spheres, except for $S^3$, can be realized as a homogeneous space of the form $\widetilde{PSL}(2,\bbr)/\grG$ where $\widetilde{PSL}(2,\bbr)$ denotes the universal cover of the projective linear group $PSL(2,\bbr)$, and $\grG$ is a cocompact discrete subgroup of $\widetilde{PSL}(2,\bbr)$ \cite{NeRa78,RaVa81,LeRa10}. Hence, $\pi_1(M^3)$ is infinite and $M^3$ is aspherical. Recall that a manifold $M$ is {\it aspherical} if $\pi_k(M)=0$ for all $k>1$. Summarizing we have

\begin{proposition}\label{M3asph}
Let $M^3$ be a Seifert fibered homology 3-sphere which is not the standard sphere nor the Poincar\'e sphere. Then $M^3$ is aspherical and a homogeneous space of the form $\widetilde{\rm PSL}(2,\bbr)/\grG$ for some cocompact infinite discrete subgroup $\grG$. Furthermore, $\pi_1(M^3)$ is infinite and perfect.
\end{proposition}

\subsection{The Orbifold Base}

We consider the homology 3-sphere as the Seifert bundle $S^1\ra{1.5} M^3\ra{1.5} B$ where $B$ is an orbifold. As will become evident below in all cases the base of the Seifert fibration is $S^2$ with an orbifold structure. We refer the reader to Chapter 4 of \cite{BG05} for the basics of orbifolds. In order to work with orbifold cohomology classes we consider Haefliger's \cite{Hae84} classifying space $\mathsf{B}B$ of an orbifold $B$. We represent the orbifold $B$ by a proper \'etale Lie groupoid $\calg$ (or any Morita equivalent Lie groupoid) and let $\mathsf{B}B$ denote the classifying space of $\calg$, then the integral orbifold cohomology groups and homotopy groups are defined by $H^i_{orb}(B,\bbz)=H^i(\mathsf{B}B,\bbz)$ and $\pi_i^{orb}(B,\bbz)=\pi_i(\mathsf{B}B,\bbz)$, see Section 4.3 of \cite{BG05} for further details. 

We have

\begin{lemma}\label{piorb} 
Let $M^3$ be a Seifert fibered homology 3-sphere that is not the standard sphere $S^3$. Then $\pi_1^{orb}(B)$ is a quotient of $\pi_1(M^3)$, and hence, perfect. If $M^3$ is not the Poincar\'e sphere, then $\pi_i^{orb}(B)=0$ for $i\geq 2$ and $\pi_1^{orb}(B)\approx \pi_1(M^3)/\bbz$. If $M^3=S^3/\bbi^*$, then $\pi_2^{orb}(B)\approx \bbz$, $\pi_i^{orb}(B)\approx \pi_i(S^3/\bbi^*)\approx \pi_i(S^3)$ for $i\geq 3$, and $\pi_1^{orb}(B)$ is $\bbi$.
\end{lemma}

\begin{proof}
Since in all cases $\pi_2(M^3)=0$ we have the homotopy exact sequence
\begin{equation}\label{3sphexhom}
0\ra{1.8}\pi_2^{orb}(B)\fract{\gri}{\ra{1.8}} \bbz\fract{\psi}{\ra{1.8}} \pi_1(M^3)\ra{1.8} \pi_1^{orb}(B)\ra{1.8} \{1\}
\end{equation}
which proves the first statement. If $M^3$ is not the Poincar\'e sphere then $\pi_1(M^3)$ is infinite, so by Lemma 14.3.1 of \cite{LeRa10} $\psi$ is injective which proves the second statement. If $M^3$ is the Poincar\'e sphere then $\pi_1(M^3)\approx \bbi^*$ and Lemma 14.3.1 of \cite{LeRa10} says that $B$ is $S^2$. From the exact sequence (\ref{3sphexhom}) and the fact that $\pi_1(M^3)=\bbi^*$ we see that $\pi_1^{orb}(B)$ is either $\bbi^*$ or $\bbi$. We claim that it must be $\bbi$. The following argument is taken from \cite{Zim11}. The orbifold structure of $B$ is described by the branched cover $S^2\ra{1.7}S^2/G$ where $G$ is either $\bbi$ or $\bbi^*$. Now as we shall see shortly the Poincar\'e sphere can be represented by the link $L(2,3,5)$ and  the orbifold Riemann-Hurwitz formula is
$$\chi(S^2)=|G|\bigl(2-2g-\sum_i(1-\frac{1}{m_i})\bigr)$$
where $\chi$ is Euler characteristic and $g$ is the genus of the Riemann surface, so we have $\chi(S^2)=2$, $g=0$, and $(m_1,m_2,m_3)=(2,3,5)$. Thus, the formula gives $|G|=60$, so $G=\bbi$. Then from the exact sequence (\ref{3sphexhom}) we see that $\pi_2^{orb}(B)\approx \bbz$ and the map $\gri$ is multiplication by $2$.
\end{proof}

In all cases in this paper the 2-dimensional orbifold $B$ is {\it developable}, that is, it is a global quotient, namely, $B=S^2/\pi_1^{orb}(B)$.

A result of Neumann and Raymond \cite{NeRa78} says that up to orientation any Seifert fibered homology 3-sphere can be realized as a link of complete intersections of generalized Brieskorn manifolds \cite{Ran75}. Moreover, it follows from Section 9.6 of \cite{BG05} that any such link admits a Sasakian structure. 
Seifert fibrations as complete intersections is also treated in \cite{Sav02,Loo84} as well as recently in Section 14.11 of \cite{LeRa10}.

\subsection{Complete Intersections of Brieskorn Manifolds}\label{secbrhomsph}
Here we essentially follow Section 14.11 of \cite{LeRa10}. Let $\bfa=(a_0,\cdots,a_n)$ be a sequence of integers with $a_i\geq 2$ and let $C=(c_{ij})$ be an $n-1$ by $n+1$ matrix of complex numbers. We also assume that each $n-1$ by $n-1$ minor determinant of $C$ is nonzero. Then the complex variety 
\begin{equation}\label{genBrvar}
V_C(\bfa) =\{z\in \bbc^{n+1}~|~f_i=c_{i0}z_0^{a_0}+\cdots +c_{in}z_n^{a_n}=0,~i=0,\cdots,n-2\}
\end{equation}
is nonsingular away from the origin in $\bbc^{n+1}$, and the link
\begin{equation}\label{genBrlnk}
L(\bfa)=V_C(\bfa)\cap S^{2n+1}
\end{equation}
is a smooth 3-dimensional manifold which is independent of $C$ up to diffeomorphism. Note that $V_C$ has a $\bbc^*$ action defined by $z\mapsto (\grl^{w_0}z_0,\cdots,\grl^{w_n}z_n)$ with weights 
$$w_j=\frac{\lcm(a_0,\cdots,a_n)}{a_j}.$$
The unit circle $S^1\subset \bbc^*$ acts on the link $L(\bfa)$ without fixed points, so it is a Seifert manifold.
The Seifert invariants are given on page 336 of \cite{LeRa10}. The unnormalized invariants are 
$$M=\{\go,g,0,0,s_1(\gra_0,\grb_0),\ldots,s_n(\gra_n,\grb_n\}$$ 
where
$$\gra_j=\frac{\lcm(a_0,\cdots,a_n)}{\lcm(a_0,\cdots,\widehat{a}_j,\cdots,a_n)},\quad s_j= \frac{a_0a_1\cdots\widehat{a}_j\cdots a_n}{\lcm(a_0,\cdots,\widehat{a}_j,\cdots,a_n)}$$
$$g=\frac{1}{2}\bigl(2+(n-1)\frac{a_0a_1\cdots a_n}{\lcm(a_0,\cdots,a_n)} -\sum_{j=0}^ns_j\bigr)$$
and $\grb_j$ and $e$ are determined by 
$$-e(M)=\sum_{j=0}^ns_j\frac{\grb_j}{\gra_j} =\frac{a_0a_1\cdots a_n}{\lcm(a_0,\cdots,a_n)^2}.$$
This last equation becomes 
\begin{equation}\label{beqn}
\sum_{j=0}^n\grb_jw_j=1.
\end{equation}
Here $g$ is the genus of the base Riemann surface. Note that the variety $V_C$ is not the most general type of complete intersection even for dimension 3. Generally, complete intersections have a multidegree $\bfd=(d_0,\cdots,d_{n-2})$, but in our case here the degrees $d_i$ are all equal, namely $d=\lcm(a_0,\cdots,a_n)$. 
We define the weight vector $\bfw=(w_0,\cdots,w_n)$ and $|\bfw |=\sum_j w_j$.  
The link $L(a_0,\ldots,a_n)$ is the total space of an $S^1$ orbibundle over a projective algebraic orbifold $\calz$, and $e(M)$ is the orbifold Euler number of the orbibundle. Note that the Euler number of the orbifold canonical bundle of this orbifold is  $\frac{d-|\bfw|}{d}$. We now consider the main case of our interest, namely, \cite{NeRa78}

\begin{lemma}\label{homsph}
The link $L(\bfa)$ is a homology 3-sphere if and only if the integers $a_0,\ldots,a_n$ are pairwise relatively prime.
\end{lemma}

In this case the weights are $w_j=a_0\cdots \widehat{a}_j\cdots a_n$ and the degree $d=a_0\cdots a_n$ which satisfies $d=w_ja_j$ for all $j$. (This last equation holds in generally for Brieskorn manifolds).
So in the case of homology spheres we have 
\begin{equation}\label{invhomsph}
\gra_j=a_j,\quad s_j=1,\quad g=0. \quad -e(M) =\frac{1}{a_0a_1\cdots a_n} =\frac{1}{d}.
\end{equation}
The equation for $\grb_j$ becomes
$$\sum_{j=0}^n\frac{\grb_j}{a_j}=\frac{1}{a_0a_1\cdots a_n}.$$
So for homology 3-spheres we can take 
$$\grb_j=\frac{a_j}{(n+1)d}.$$

As mentioned previously every Seifert fibered manifold that is an integer homology sphere is diffeomorphic up to orientation to a Brieskorn complete intersection $L(a_0,\ldots,a_n)$ \cite{NeRa78}. Furthermore, they are all homogeneous manifolds which except for the Poincar\'e sphere have the form $\widetilde{PSL}(2,\bbr)/\grG$ where $\grG$ is an infinite discrete group such that $[\grG,\grG]=\grG$. The Poincar\'e sphere is represented by $\bfa=(2,3,5)$. It is the homogeneous space $SU(2)/\bbi^*$ where $\bbi^*$ is the binary icosahedral group, a finite group of order $120$. Conditions on the Seifert invariants that a given Seifert manifold can be represented as $\widetilde{PSL}(2,\bbr)/\grG$ where $\grG$ is a cocompact discrete subgroup of $\widetilde{PSL}(2,\bbr)$ were given in \cite{RaVa81}. Also the fact that any complete intersection link of the form (\ref{genBrlnk}) is a homogeneous space $\widetilde{PSL}(2,\bbr)/\grG$ or $SU(2)/\bbi^*$ was proved in \cite{Neu77}.

\begin{example}\label{n2n3}
When $n=2$ we have the Brieskorn 3-manifolds $L(a_0,a_1,a_2)$ which were treated extensively in \cite{Mil75}.
When $n=3$ we can redefine the coordinates giving the complete intersection
$$z_0^{a_0}+z_1^{a_1}+z_2^{a_2}+z_3^{a_3}=0,\qquad c_0z_0^{a_0}+c_1z_1^{a_1}+c_2z_2^{a_2}+c_3z_3^{a_3}=0.$$
The condition on the minor determinants is $c_i\neq c_j$ for all $i\neq j=0,\cdots, 3$.
\end{example}

\subsection{Sasakian Structures on Homology Spheres}
It follows from Proposition 9.6.1 of \cite{BG05} that the links $L(\bfa)$ admit a quasiregular Sasakian structure induced from the Sasakian structure on $S^{2n+1}$ as a complete intersection. Now $L(\bfa)$ is the total space of an $S^1$ orbibundle over a projective algebraic variety $\calz_\bfw$ of complex dimension one, that is, a Riemann surface embedded as a complete intersection in the weighted projective space $\bbc\bbp(\bfw)$. Given the form of the weights for homology spheres, we see that $\bbc\bbp(\bfw)$ is isomorphic as a projective variety to $\bbc\bbp^n$ and $\calz_\bfw$ is isomorphic to $\bbc\bbp^1$ but with a non-trivial orbifold structure.

Recall from Definition 7.5.24 of \cite{BG05} that a Sasakian structure $\cals=(\xi,\eta,\Phi,g)$ is {\it positive (negative)} if the basic first Chern class $c_1^B(\calf_\xi)$ can be represented by a positive (negative) definite $(1,1)$ form.

\begin{lemma}\label{sastype}
Let $L(\bfa)$ be a homology 3-sphere. Then except for the Poincar\'e sphere with link $L(2,3,5)$, the Sasakian structure on the link $L(\bfa)$ is negative. The Sasakian structure on the Poincar\'e sphere $L(2,3,5)$ is positive.
\end{lemma}

\begin{proof}
Note that the basic first Chern class $c_1^B(\calf_\xi)$ is the pullback of the orbifold first Chern class which is $\frac{|\bfw |-|\bfd |}{d}$ times the area form $\gra$ on $\bbc\bbp^1$. Thus, we have
\begin{equation}\label{c1orb}
c_1^{orb}(\calz_\bfw)=\frac{1}{d}\bigl(\sum_jw_j-(n-1)a_0\cdots a_n\bigr)\gra.
\end{equation}
If $n=2$ then it is well known \cite{Mil75} and easy to see that $L(2,3,5)$ is the only non-standard homology sphere for which $c_1>0$. So we assume that $n\geq 3$ and without loss of generality we can assume that $a_0<a_1 <\cdots <a_n$ and $a_j\geq 2$. We then have
\begin{eqnarray*}
|\bfd |-|\bfw |&=&(n-1)a_0\cdots a_n-\sum_ja_0\cdots \widehat{a}_j\cdots a_n \\
  &>& (n-1)a_0\cdots a_n -(n+1)a_1\cdots a_n \\
 &=& \bigl((n-1)a_0-(n+1)\bigr)a_1\cdots a_n \geq (n-3)a_1\cdots a_n\geq 0.
\end{eqnarray*}
\end{proof}

When $M^3$ is a homology sphere with infinite fundamental group arising from the link $L(\bfa)$, the base $B=B(\bfa)$ is $S^2$ with an orbifold structure consisting of $n+1$ orbifold points. So if $M^3$ is not the Poincar\'e sphere, or equivalently that $\pi_1(M^3)$ is infinite, there are $n+1$ singular orbits corresponding to setting $z_j=0$ where $j=0,\ldots,n$. The order of the isotropy subgroup when $z_j=0$ is $a_j$, so the total order of the orbifold quotient is $d_\bfa=a_0\cdots a_n$ which coincides with the order $\upsilon$ of the Sasakian structure.

We are now ready for Belgun's theorem. The version given in Theorem 10.1.3 of \cite{BG05} is more convenient for our purposes, and we give only what we need here.  
\begin{theorem}[\cite{Bel01}]\label{Belthm}
Let $M$ be a $3$-dimensional compact manifold admitting a Sasakian
structure $\cals=(\xi,\eta,\Phi,g).$ Then
\begin{enumerate}
\item If $\cals$ is positive, $M$ is spherical, and there is a
Sasakian metric of constant $\Phi$-sectional curvature $1$ in the
same deformation class as $g.$ 
\item If $\cals$ is negative, $M$ is of $\widetilde{PSL}(2,\bbr)$ type, and there is a Sasakian metric of constant $\Phi$-sectional curvature $-4$ in the same deformation
class as $g.$ 
\end{enumerate}
\end{theorem}

In the positive case the metric $g$ also has constant Riemannian sectional curvature $1$. Moreover, in case 1 of the theorem there is a ray of constant $\Phi$ sectional curvature $c$ with $c>-3$ which corresponds to constant transverse holomophic sectional curvature $c+3>0$. In case 2 there is a ray of constant $\Phi$ sectional curvature $c$ with $c<-3$ which corresponds to constant transverse holomorphic sectional curvature $c+3< 0$. Except for the standard sphere the automorphism group of $\cals$ is one dimensional consisting of the flow generated by the Reeb vector field. Thus, the Sasaki cone is one dimensional for these structures.

\subsection{Orbifold K\"ahler-Einstein Metrics on $B_\bfa$}\label{orbKE}

It follows from Theorem \ref{Belthm} and the orbifold Boothby-Wang Theorem \cite{BG00a} that $B_\bfa$ admits a K\"ahler-Einstein orbifold metric. For example for the Poincar\'e sphere with $\bfa=(2,3,5)$, the orbifold first Chern class of the base orbifold $B_\bfa$ is $c_1^{orb}(B_\bfa)=\frac{1}{30}\gra$ where $\gra$ is primitive in $H^2(B_\bfa,\bbz)$. So if $p:\mathsf{B}B_\bfa\ra{1.6} B_\bfa$ is the orbifold classifying map, $p^*c_1^{orb}$ is primitive in $H^2_{orb}(B_\bfa,\bbz)$. Moreover, $c_1^{orb}$ pulls back to an integer cohomology class on $M^3_\bfa$ which of course in our case is $0$.

We want to emphasize, as mentioned previously, that the orbifold K\"ahler-Einstein metrics on $B_\bfa$ correspond to Sasaki-$\eta$-Einstein metrics on $M^3_\bfa$. In the positive case when $M^3_\bfa=S^3/\bbi^*$ and $\bfa=(2,3,5)$, the positive K\"ahler-Einstein metric on $B_{(2,3,5)}$ with scalar curvature $8$ corresponds to the standard constant sectional curvature (and $\Phi$ sectional curvature) $1$ metric on $S^3/\bbi^*$. In the negative case the K\"ahler-Einstein metric on $B_\bfa$ with scalar curvature $-2$ corresponds to constant $\Phi$ sectional curvature $-4$ on $M^3_\bfa$. In both cases by rescaling the K\"ahler-Einstein metric on $B_a$, we obtain a ray of Sasaki-$\eta$-Einstein metrics on $M^3_\bfa$ with constant $\Phi$ sectional curvature. 

In order to construct $\eta$-Einstein metrics on joins, we need to consider an {\it index}. In \cite{BG00a} we worked with positive Sasakian structure in which case we defined the Fano index. For complete intersections it is $|\bfw|-|\bfd|$. However, here in all but one case, the Sasakian structure is negative, so we define the {\it canonical index} as $I_\bfa=|\bfd|-|\bfw|$. For homology spheres $M^3_\bfa$, $|\bfd|=(n-1)a_0\cdots a_n$ and $|\bfw|=\sum_ja_0\cdots \hat{a_j}\cdots a_n$.

\section{The Join of an Homology 3-Sphere}

Here we let $M^3_\bfa$ be a homology sphere described by the links $L(a_0,\ldots,a_n)$ of Section \ref{secbrhomsph} and we assume that $a_i\geq 2$ for all $i$, so that it is not the standard sphere. As seen in the last section $M^3_\bfa$ has a natural Sasakian structure of constant $\Phi$ sectional curvature, and we can easily obtain higher dimensional Sasakian manifolds with perfect fundamental group by applying the join construction. 

\subsection{The Topology of $M^3_\bfa\star_{k,l}N$ with $N$ Simply Connected}
We begin with
\begin{theorem}\label{1conperfect}
Let $N$ be a simply connected $2r+1$-dimensional quasi-regular Sasakian manifold of order $\upsilon$, and let $M^3_\bfa$ be an homology sphere described in Section \ref{secbrhomsph}, so $\gcd\{a_i\}=1$. If also $\gcd(la_0\cdots a_n,k\upsilon)=1$ then the $(k,l)$-join $M^3_\bfa\star_{k,l}N$ is a $2r+3$-dimensional Sasakian manifold. Moreover, if $M^3_\bfa$ is not the Poincar\'e sphere we have
\begin{itemize}
\item $M^3_\bfa\star_{k,l}N$ has a perfect fundamental group isomorphic to a quotient of $\pi_1(M^3_\bfa)$.
\item $\pi_1(M^3_\bfa\star_{k,l}N)$ is a $\bbz_l$ extension of the perfect group $\pi_1^{orb}(B_\bfa)\approx \pi_1(M^3_\bfa)/\bbz$.
\item $\pi_i(M^3_\bfa\star_{k,l}N)\approx \pi_i(N)~~\text{for $i\geq 2$}$.
\end{itemize}

If $M^3_\bfa$ is the Poincar\'e sphere $S^3/\bbi^*$, then 
\begin{itemize}
\item If $l$ is odd $M^3_\bfa\star_{k,l}N$ has a perfect fundamental group equal to $\bbi$.
\item If $l$ is even $M^3_\bfa\star_{k,l}N$ has a perfect fundamental group equal to either $\bbi$ or $\bbi^*$.
\item $\pi_i(M^3_\bfa\star_{k,l}N)\approx \pi_i(S^3)\oplus \pi_i(N) ~~\text{for all $i\geq 3$}.$
\end{itemize}
\end{theorem}

\begin{proof}
The fact that $M^3_\bfa\star_{k,l}N$ is a smooth Sasakian manifold of dimension $2r+3$ follows from Proposition 7.6.6 of \cite{BG05}. There are two relevant fibrations involving the join. By Proposition 7.6.7 of \cite{BG05}  $M^3_\bfa\star_{k,l}N$ can be realized as a $N/\bbz_l$-bundle over the base orbifold $B_\bfa$ of $M^3_\bfa$, that is we have an orbifold fibration
\begin{equation}\label{prop767}
N/\bbz_l\ra{1.8} M^3_\bfa\star_{k,l}N\ra{1.8} B_\bfa
\end{equation}
whose long exact homotopy sequence is
\begin{equation}\label{homexactseq}
0\ra{1.5}\pi_2(N)\ra{1.5} \pi_2(M^3_\bfa\star_{k,l}N)\ra{1.5}\pi_2^{orb}(B_\bfa) \ra{1.5}\bbz_l\ra{1.5} \pi_1(M^3_\bfa\star_{k,l}N) \ra{1.5} \pi_1^{orb}(B_\bfa)\ra{1.5} 1.
\end{equation}
On the other hand we have the circle bundle that defines the join, namely
\begin{equation}\label{joinbundle}
S^1\ra{1.5} M^3_\bfa\times N\ra{1.6} M^3_\bfa\star_{k,l}N,
\end{equation} 
which gives the long exact sequence
\begin{equation}\label{longexact}
0\ra{1.6}\pi_2(M^3_\bfa\times N)\ra{1.6}\pi_2(M^3_\bfa\star_{k,l}N)\ra{1.6}\bbz\ra{1.6} \pi_1(M^3_\bfa)\ra{1.6} \pi_1(M^3_\bfa\star_{k,l}N)\ra{1.6} 1.
\end{equation}
Since the quotient of a perfect group is perfect, this implies that in all cases the fundamental group of $M^3_\bfa\star_{k,l}N$ is perfect.

Now if $M^3_\bfa$ is not the Poincar\'e sphere Lemma \ref{piorb} says that for $i\geq 2$, $\pi_i^{orb}(B_\bfa)=0$ which implies that $\pi_1(M^3_\bfa\star_{k,l}N)$ is a $\bbz_l$ extension of $\pi_1^{orb}(B_\bfa)$. The second statement then follows immediately. The third statement also follows from the exact homotopy sequence and the fact that $\pi_i^{orb}(B_\bfa)=0$ for $i\geq 2$.

When $M^3_\bfa$ is the Poincar\'e sphere $S^3/\bbi^*$, the long exact homotopy sequence (\ref{homexactseq}) becomes, again using Lemma \ref{piorb},
\begin{equation}\label{homexactseq2}
\ra{1.8}\pi_2(N)\ra{1.8} \pi_2(M^3_\bfa\star_{k,l}N)\ra{1.8} \bbz\ra{1.8}\bbz_l\ra{1.8} \pi_1(M^3_\bfa\star_{k,l}N) \ra{1.8} \bbi\ra{1.8} 1.
\end{equation}
But also in this case the long exact homotopy sequence (\ref{longexact}) gives
\begin{equation}\label{homexactseq3}
	0\ra{1.8}\pi_2(N)\ra{1.8} \pi_2(M^3_\bfa\star_{k,l}N)\ra{1.8} \bbz\ra{1.8}\bbi^*\ra{1.8} \pi_1(M^3_\bfa\star_{k,l}N) \ra{1.8} 1.
\end{equation}
From the exact sequence (\ref{homexactseq3}), $\pi_1(M^3_\bfa\star_{k,l}N)$ is either $\bbi$ or $\bbi^*$. But from (\ref{homexactseq2}) if $l$ is odd, then $\pi_1(M^3_\bfa\star_{k,l}N)$ cannot be $\bbi^*$ which proves the first statement of this case.  However, if $l$ is even, $\pi_1(M^3_\bfa\star_{k,l}N)$ can be either $\bbi$ or $\bbi^*$. The final statement follows from the homotopy exact sequence of (\ref{joinbundle}).
\end{proof}

Also when $l$ is odd the map in (\ref{homexactseq3})
$$\pi_2(M^3_\bfa\star_{k,l}N)/\pi_2(N)\approx \bbz\ra{1.8} \bbz$$ 
is multiplication by $2$; whereas, the similar map in (\ref{homexactseq2}) is multiplication by $l$.

\begin{remark} 
Notice that for fixed $\bfa$ and $\upsilon$, there are infinitely many $(k,l)$ that satisfy the smoothness condition $\gcd(la_0\cdots a_n,k\upsilon)=1$.
\end{remark}

For cohomology we have
\begin{lemma}\label{Hjoin}
Assume the hypothesis of Theorem \ref{1conperfect}. 
Then 
$H^2(M^3_\bfa\star_{k,l}N,\bbz)\approx H^2(N,\bbz)\oplus \bbz.$
\end{lemma}

\begin{proof}
Now $M^3_\bfa$ is a homology sphere, and by Theorem \ref{1conperfect} $M^3_\bfa\star_{k,l}N$ has perfect fundamental group, so the Gysin sequence of the fibration (\ref{joinbundle}) gives the short exact sequence
$$0\ra{1.8}\bbz \ra{1.8} H^{2}(M^3_\bfa\star_{k,l} N,\bbz)\ra{1.8} H^2(N,\bbz) \ra{1.8}0$$
which splits since $H^2(N,\bbz)$ is free.
\end{proof}

A special case of interest is $N=S^{2r+1}$, the $(2r+1)$-sphere. 
\begin{proposition}\label{homnsphe}
Let $M^3_\bfa$ be an homology 3-sphere. Then the join $M^3_\bfa\star_{k,l}S^{2r+1}$ has the rational cohomology ring of $S^2\times S^{2r+1}$ for all relatively prime positive integers $k,l$, and the join $M^3_\bfa\star_{k,1}S^{2r+1}$ has the integral cohomology ring of $S^2\times S^{2r+1}$ for all positive integers $k$. If $r=1$, the join $M^3_\bfa\star_{k,l}S^3$ has the integral cohomology ring of $S^2\times S^3$ for all relatively prime positive integers $k,l$.
\end{proposition}

\begin{proof}
Our proof uses the spectral sequence method employed in \cite{WaZi90,BG00a} (see also Section 7.6.2 of \cite{BG05}). The fibration (\ref{joinbundle}) together with the torus bundle with total space $M^3_\bfa\times S^{2r+1}$ gives the commutative diagram of fibrations
\begin{equation}\label{orbifibrationexactseq}
\begin{matrix}M^3_\bfa\times S^{2r+1}&\ra{2.6} &M^3_\bfa\star_{k,1}S^{2r+1}&\ra{2.6}
&\mathsf{B}S^1 \\
\decdnar{=}&&\decdnar{}&&\decdnar{\psi}\\
M^3_\bfa\times S^{2r+1}&\ra{2.6} &\mathsf{B}B_\bfa\times \bbc\bbp^r&\ra{2.6}
&\mathsf{B}S^1\times \mathsf{B}S^1\, 
\end{matrix} \qquad \qquad
\end{equation}
where $\mathsf{B}G$ is the classifying space of a group $G$ or Haefliger's classifying space \cite{Hae84} of an orbifold if $G$ is an orbifold. Note that by definition $H^i_{orb}(B_\bfa,\bbz)=H^i(\mathsf{B}B_\bfa,\bbz)$ and an easy argument of the orbifold fibration $S^1\ra{1.3}M^3_\bfa\ra{1.3}B_\bfa$ shows that $H^*_{orb}(B_\bfa,\bbz)=H^*(S^2,\bbz)$. So cohomologically, $\mathsf{B}B_\bfa$ is $S^2$. We also note that $\mathsf{B}S^1=\bbc\bbp^\infty$ with cohomology ring $\bbz[s]$ with $s\in H^2(\bbc\bbp^\infty,\bbz)$.

Now the map $\psi$ is that induced by the inclusion $e^{i\theta}\mapsto (e^{il\theta},e^{-ik\theta})$. So writing $$H^*(\mathsf{B}S^1\times \mathsf{B}S^1,\bbz)=\bbz[s_1,s_2]$$ 
we see that $\psi^*s_1=ls$ and $\psi^*s_2=-ks$. The $E_2$ term of the Leray-Serre spectral sequence of the top fibration of diagram (\ref{orbifibrationexactseq}) is $E^{p,q}_2=H^p(\mathsf{B}S^1,H^q(M^3_\bfa\times S^{2r+1},\bbz))$. The non-zero terms occur when $p$ is even, say $2p'$, and $q=0,3,2r+1,2r+4$. Let $\gra,\grb$ denote the orientation class of $M^3_\bfa,S^{2r+1}$, respectively. The cohomology ring of the fiber is $\grL[\gra,\grb]$, whereas,  that of the base is $\bbz[s]$. Then if $r>1$ the differential $d_4(\gra)=(ls)^2$, so we need $l=1$ to avoid torsion in $H^4\approx$ torsion in $H_3$. Then by naturality we have $d_4(\gra\otimes s^{2p'})=s^{2p'+2}$, but then $d_{2r+2}(\grb)=0$ and $d_{2r+2}(\gra\cup\grb)=s^2\otimes\grb$, so the classes that survive are $s,\grb,\grb\cup s$ which proves the result for $r>1$. 

When $r=1$ we have $d_4(\gra)=l^2s^2$ and $d_4(\grb)=k^2s^2$ which implies that the primative integral class $k^2\gra-l^2\grb\in E_2^{0,3}$ survives, and this proves the result together with Poincar\'e duality.
\end{proof}

\subsection{The Topology of $M^3_\bfa\star_{k,l}M^3_\bfb$}
Here we consider the join of two homology 3-spheres. We assume that $\gcd(\bfa,\bfb)=1$ which implies that at most one of them can be the Poincar\'e sphere. From the fibration (\ref{joinbundle}) with $N$ replaced by $M^3_\bfb$ we obtain
$$0\ra{1.6}\pi_2(M^3_\bfa\times M^3_\bfb)\ra{1.6}\pi_2(M^3_\bfa\star_{k,l}M^3_\bfb)\ra{1.6}\bbz\ra{1.6} \pi_1(M^3_\bfa\times M^3_\bfb)\ra{1.6} \pi_1(M^3_\bfa\star_{k,l}M^3_\bfb)\ra{1.6} 1.$$
Now $M^3_\bfa\times M^3_\bfb$ has perfect fundamental group $\pi_1(M^3_\bfa)\times \pi_1(M^3_\bfb)$, and if neither is the Poincar\'e sphere $M^3_\bfa\times M^3_\bfb$ is aspherical.  Hence, $\pi_1(M^3_\bfa\star_{k,l}M^3_\bfb)$ is perfect and $\pi_i(M^3_\bfa\star_{k,l}M^3_\bfb)=0$ for $i\geq 3$. Moreover, we have
\begin{equation}\label{longexact2}
0\ra{1.6}\pi_2(M^3_\bfa\star_{k,l}M^3_\bfb)\ra{1.6}\bbz\ra{1.6} \pi_1(M^3_\bfa)\times \pi_1(M^3_\bfb)\ra{1.6} \pi_1(M^3_\bfa\star_{k,l}M^3_\bfb)\ra{1.6} 1
\end{equation}

\begin{proposition}\label{cohomjoin2}
For all relatively prime positive integers $k,l$ the 5-manifold $M^3_\bfa\star_{k,l}M^3_\bfb$ has the integral cohomology ring of $S^2\times S^3$.
\end{proposition}

\begin{proof}
The proof is similar to the $r=1$ case of Proposition \ref{homnsphe}, so the details are left to the reader.
\end{proof}

\subsection{Distinguishing Contact Structures}\label{diffcon}
The crudest invariant of a contact structure is the first Chern class $c_1(\cald)$ of the contact bundle; nevertheless, it can distinguish countably many contact structures in many cases. A much more subtle contact invariant is the contact homology of Eliashberg, Giventhal, and Hofer \cite{ElGiHo00} employed for example in \cite{BoPa10}. However, in the present paper we cannot even pin down the diffeomorphism type, so we only make use of $c_1$ to distinguish contact structures.

Here we consider only a special case where given two base orbifolds $B_1$ and $B_2$, we assume that $p^*c_1^{orb}(B_i)=-I_i\gra_i$ where $I_i$ is the canonical index\footnote{Since we mainly deal with negative Sasakian structures, we use the canonical index instead of the Fano index used in \cite{BG00a}. Of course, the Fano index is just the negative of the canonical index.} of $B_i$ and $\gra_i\in H^2_{orb}(B_i,\bbz)$ is a generator. In this case we have 
\begin{equation}\label{2c1orb}
p^*c_1^{orb}(B_1\times B_2) =-I_1\gra_1-I_2\gra_2.
\end{equation}
Now consider the $S^1$ bundle determined by the K\"ahler class $k_1\gra_1+k_2\gra_2$ on $B_1\times B_2$. Let $\pi:M_1\star_{k_1,k_2}M_2\ra{1.6} B_1\times B_2$ denote the bundle projection. By the join construction we know that $\pi^*(k_1\gra_1+k_2\gra_2)=[d\eta]=0$. So there is a generator $\grg\in H^2(M_1\star_{k_1,k_2}M_2,\bbz)$ such that $\pi^*\gra_1=k_2\grg$ and $\pi^*\gra_2=-k_1\grg$. Now the first Chern class of the contact bundle is the pullback of the orbifold first Chern class on $B_1\times B_2$, that is, 
\begin{equation}\label{2c1orb2}
c_1(\cald)=\pi^*c_1^{orb}(B_1\times B_2)=-I_1\pi^*\gra_1-I_2\pi^*\gra_2=(I_2k_1-I_1k_2)\grg.
\end{equation}
The mod 2 reduction of $c_1(\cald)$ is a topological invariant, namely the second Stiefel-Whitney class $w_2(M_1\star_{k_1,k_2}M_2)\in H^2(M_1\star_{k_1,k_2}M_2,\bbz_2)$. This allows us to distinguish different manifolds with the same perfect fundamental group.

\section{Extremal Sasaki Metrics on Joins of Homology Spheres}

We shall always assume that $M^3_\bfa$ is an homology 3-sphere with constant $\Phi$ sectional curvature either $1$ or $-4$ and that it is not the standard sphere. So it is either the Poincar\'e sphere with constant sectional curvature $1$ or a negative homology sphere with constant $\Phi$ sectional curvature $-4$. 

\begin{theorem}\label{mainthm}
Let $N$ be a simply connected quasi-regular Sasakian manifold of dimension $2r+1$ which fibers in the orbifold sense over a projective algebraic orbifold $\calz$ with an orbifold K\"ahler metric $h$ of constant scalar curvature $S$. Let $\upsilon$ be the order of the Sasakian structure on $N$ and assume that $\gcd(la_0\cdots a_n,k\upsilon)=1$. Then the join $M^3_\bfa\star_{k,l}N$ is a smooth Sasakian manifold of dimension $2r+3$ with perfect fundamental group and the induced Sasakian structure has a ray of extremal CSC Sasaki metrics.
\end{theorem}

\begin{proof}
The proof easily follows from the join construction together with Theorem \ref{1conperfect}.
\end{proof}

When $N$ has a Sasakian structure whose automorphism group $\gA\gu\gt(\cals)$ has dimension greater than 1, one can deform in the Sasaki cone to obtain new extremal Sasakian structures. Indeed, the Openness Theorem of \cite{BGS06} guarentees the existence of an open set of such extremal Sasaki metrics.

\subsection{Extremal Sasakian metrics on $M^3_\bfa\star_{k,l}S^{2r+1}$}
We can easily obtain constant scalar curvature Sasakian metrics on manifolds $M^3_\bfa\star_{k,l}S^{2r+1}$ from the lift of the product K\"ahler orbifold metric on $B(\bfa)\times \bbc\bbp^r$. Let us describe the orbifold structure on $B(\bfa)$. As an algebraic variety it is $\bbc\bbp^1$ with $n+1$ distinct marked points. Thus,  as an algebraic variety the product $B(\bfa)\times \bbc\bbp^r$ is $\bbc\bbp^1\times \bbc\bbp^r$ with a non-trivial orbifold structure on the first factor described by branch divisors $\grD=\sum_i\grD_i$. On the first factor we have an orbifold K\"ahler-Einstein metric with scalar curvature $-2$, and on the second factor the standard Fubini-Study metric with constant scalar curvature $4r(r+1)$.  Since there are an infinite number of integers $l,k,a_0,\ldots,a_n$ that satisfy $\gcd(la_0\cdots a_n,k)=1$, Theorems \ref{1conperfect} and \ref{mainthm}, and the results of Section \ref{diffcon} imply 

\begin{theorem}\label{homs2s3}
Let $M^3_\bfa$ be a negative homology sphere and assume that $\gcd(la_0\cdots a_n,k)=1$. Then there is a countably infinite number of $(2r+3)$-dimensional contact manifolds $M^3_\bfa\star_{k,l}S^{2r+1}$ with a perfect fundamental group that admit a ray of extremal CSC Sasaki metrics.  
\end{theorem}

Then Proposition \ref{homnsphe} and Theorem \ref{homs2s3} immediately give Theorem \ref{thmA} of the introduction.

\subsection{Extremal Sasakian metrics on $M^3_\bfa\star_{k,l}M^3_\bfb$}

Notice that since both homology 3-spheres have constant $\Phi$-sectional curvature, the bases have constant holomorphic sectional curve; hence, they both have CSC K\"ahler orbifold metrics. So if $\gcd(la_0\cdots a_n, kb_0\cdots b_n)=1$, the 5-manifold $M^3_\bfa\star_{k,l}M^3_\bfb$ has a CSC Sasaki metric. This includes the case when one of the homology 3-spheres is  the Poincar\'e sphere.

When both homology 3-spheres are negative, we obtain Sasaki-$\eta$-Einstein metrics in certain cases. In order to get such metrics we need to select the $S^1$ orbibundle over $B_\bfa\times B_\bfb$ whose Euler class is proportional to the orbifold first Chern class $c_1^{orb}(B_\bfa\times B_\bfb)$ with negative proportionality constant.
From Equation (\ref{c1orb}) we have with obvious notation
\begin{equation}\label{c1orb2}
c_1^{orb}(B_\bfa\times B_\bfb) =\frac{|\bfw_\bfa|-|\bfd_\bfa|}{d_\bfa}\gra + \frac{|\bfw_\bfb|-|\bfd_\bfb|}{d_\bfb}\grb.
\end{equation}
So by Equation (\ref{2c1orb2}) to obtain a Sasaki-$\eta$-Einstein metric on $M^3_\bfa\star_{k,l}M^3_\bfb$ we need 
$c_1(\cald)=\pi^*c_1^{orb}(B_\bfa\times B_\bfb)=0$, and we do this by choosing
$k=|\bfd_\bfa|-|\bfw_\bfa|=I_\bfa$ and $l=|\bfd_\bfb|-|\bfw_\bfb|=I_\bfb$ as long as they are relatively prime. As in \cite{BG00a} to handle the case when they are not relatively prime, we define the {\it relative indices} by 
$$\cali_\bfa=\frac{I_\bfa}{\gcd(I_\bfa,I_\bfb)}, \qquad \cali_\bfb=\frac{I_\bfb}{\gcd(I_\bfa,I_\bfb)}.$$
Then $\gcd(\cali_\bfa,\cali_\bfb)=1$. So generally we obtain a Sasaki-$\eta$-Einstein metric on $M^3_\bfa\star_{k,l}M^3_\bfb$ by choosing $k=\cali_\bfa$ and $l=\cali_\bfb$. We also make note of the easily shown fact that $\gcd(d_\bfa,I_\bfa)=1$ for all $\bfa$ with the $a_i$ pairwise relatively prime. This guarentees that as long as $\gcd(d_\bfa,d_\bfb)=1$, there is a Sasaki-$\eta$-Einstein metric on $M^3_\bfa\star_{\cali_\bfa,\cali_\bfb}M^3_\bfb$. Summarizing we have

\begin{theorem}\label{MaMb}
Let $M^3_\bfa$ and $M^3_\bfb$ be negative homology 3-spheres with canonical indices $I_\bfa$ and $I_\bfb$, respectively. Then 
\begin{enumerate}
\item If $\gcd(ld_\bfa,kd_\bfb)=1$ the 5-dimensional contact manifold $(M^3_\bfa\star_{k,l}M^3_\bfb,\cald)$ with $c_1(\cald)=(kI_\bfb-lI_\bfa)\grg$ admits a CSC Sasaki metric.
\item If $\gcd(d_\bfa,d_\bfb)=1$, the 5-dimensional contact manifold $(M^3_\bfa\star_{\cali_\bfa,\cali_\bfb}M^3_\bfb,\cald)$ with $c_1(\cald)=0$ admits a ray of negative Sasaki-$\eta$-Einstein metrics; hence, it also admits Lorentzian Sasaki-Einstein metrics.
\end{enumerate}
\end{theorem}

Since for fixed $\bfa$ and $\bfb$ there are a countably infinite number of relatively prime pairs $(k,l)$ that satisfy $\gcd(ld_\bfa,kd_\bfb)=1$, there are a countably infinite number of such contact 5-manifolds.

\subsection{More Sasaki-$\eta$-Einstein Manifolds}

It is now straightforward to construct higher dimensional examples of $\eta$-Einstein and Lorentzian Sasaki-Einstein metrics. Since as mentioned above $\gcd(d_\bfa,I_a)=1=\gcd(\cali_\bfa,\cali)$, we have

\begin{theorem}\label{negetaEin}
Let $M^3_\bfa$ be a negative homology 3-sphere and let $N$ be a simply connected negative Sasakian manifold with canonical index $I$ and order $\upsilon$. Let $\cali_\bfa,\cali$ denote the relative canonical indices. Then if $\gcd(d_\bfa \cali,\upsilon)=1$ the manifold $M^3_\bfa \star_{\cali_\bfa,\cali} N$ has perfect fundamental group and admits a ray of negative Sasaki-$\eta$-Einstein metrics, and hence, a Lorentzian Sasaki-Einstein metric.
\end{theorem}

\begin{proof}[Proof of Theorem \ref{sasetaein}] We need to find simply connected negative eta-Einstein 5-manifolds with arbitrary second Betti number $b_2$ that satisfy the hypothesis of Theorem \ref{negetaEin}. Then by Lemma \ref{Hjoin} we will have arbitrary $b_2(M^3_\bfa \star_{\cali_\bfa,\cali} N)=b_2(N)+1\geq 1$. In order to find such $N$ we make use of the work of Gomez \cite{Gom11} on negative $\eta$-Einstein structures on arbitrary connected sums of $S^2\times S^3$. To exhaust all connected sums $l(S^2\times S^3)$ we divide the analysis into seven cases, two infinite series and five sporatic cases. In each case we need only to check that $\gcd(I,\upsilon)=1$. Then we can always find $\bfa$ such $\gcd(d_\bfa,\upsilon)=1$.  The first series is given by the link of the hypersurface $z_0^4+z_1^{8k+2} +z_2^{4k+1}z_3+z_3^{2k+1}z_2=0.$ Here $k\geq 1$ and the second Betti number $b_2$ is $l=2k+1$. One checks that this has $\upsilon=2(4k+1)$ and $I=16k(4k+1)-(4k+3)^2$. It follows that $\gcd(I,\upsilon)=1$ in this case. The second series applies to $b_2=k-1$ with $k\geq 9$ and odd. Here the hypersurface is $z_0^4+z_1^2+z_2^k+z_3^k=0$ which has $I=k-8$ and $\upsilon =2k$. One easily sees that $\gcd(I,\upsilon)=1$ in this case. The remaining five cases are $l=0,1,2,4,6$. The last two are represented in \cite{Gom11}; however, their weights and indices are quite large, but more importantly, the $l=4$ case has $\gcd(I,\upsilon)=11$, so the join will not be smooth. The case $l=6$ in \cite{Gom11} does satisfy this condition, so it can be used. Nevertheless, in the table below we give simpler polynomials to cover these five sporatic cases. We used Orlik's formula (see Corollary 9.3.13 in \cite{BG05}) with a Maple program to compute the second Betti number $b_2$. Note that they are not necessarily connected sums of $S^2\times S^3$, as there may be torsion in $H_2(N,\bbz)$. 

\bigskip
\centerline{Sporatic Cases with $I=1$}
\begin{center}
\begin{tabular}{| l || l | l |}
\hline
$b_2$  & $~~~~~\bfw$ & ~~~Polynomial \\ \hline 
0 & (5,6,6,8) & $z_0^8+z_1^4+z_2^4+z_3^3$ \\ \hline
1 & (2,4,6,11) & $z_0^{12}+z_1^6+z_2^4+z_3^2z_0$ \\ \hline
2 & (6,7,28,42) & $z_0^{14}+z_1^{12}+z_2^3+z_3^2$ \\ \hline
4 & (4,5,20,30) & $z_0^{15}+z_1^{12}+z_2^3+z_3^2$ \\ \hline
6 & (3,4,12,16) & $z_0^{12}+z_1^9+z_2^3+z_3^2z_1$ \\ \hline

\end{tabular}
\end{center}
\medskip
\end{proof}

\subsection{Sasaki-Einstein Metrics on Manifolds with Perfect Fundamental Group}

Since the only positive homology 3-sphere with perfect fundamental group is the Poincar\'e sphere $S^3/\bbi^*$, only joins with $S^3/\bbi^*$ can give Sasaki-Einstein metrics. Furthermore, since from Belgun's Theorem \ref{Belthm} the positive case corresponds to the bi-invariant constant sectional curvature $1$ metric, these can all be obtained as quotients by $\bbi^*$ of the join of the standard $S^3$ with any simply connected Sasaki-Einstein manifold. So we can consider all the examples in \cite{BG00a} which involve a join of $S^3$ with a simply connected Sasaki-Einstein manifold. We only need to choose $k,l$ to obtain a monotone $S^1$ orbibundle, that is, its cohomology class is primitive and proportional to the first Chern class of the orbifold anti-canonical bundle $K^{-1}_{orb}(\calz)$ with $\calz=\calz_1\times \calz_2$. We also need to scale the orbifold K\"ahler metric $h$ on $\calz$ so that its scalar curvature is $4(n_1+n_2)(n_1+n_2+1)$ where the orbifold $\calz_i$ has complex dimension $n_i$. 

The Poincar\'e sphere can be represented by the link of the polynomial $z_0^5+z_1^3+z_2^2$ with weight vector $\bfw=(6,10,15)$ and degree $d=30$. It has Fano index $I_F=|\bfw|-d=31-30=1$. So using Theorems \ref{1conperfect} and \ref{homnsphe} we obtain Theorem \ref{sethm} of the Introduction.

\def\cprime{$'$} \def\cprime{$'$} \def\cprime{$'$} \def\cprime{$'$}
  \def\cprime{$'$} \def\cprime{$'$} \def\cprime{$'$} \def\cprime{$'$}
  \def\cdprime{$''$} \def\cprime{$'$} \def\cprime{$'$} \def\cprime{$'$}
  \def\cprime{$'$}
\providecommand{\bysame}{\leavevmode\hbox to3em{\hrulefill}\thinspace}
\providecommand{\MR}{\relax\ifhmode\unskip\space\fi MR }
\providecommand{\MRhref}[2]{%
  \href{http://www.ams.org/mathscinet-getitem?mr=#1}{#2}
}
\providecommand{\href}[2]{#2}

\end{document}